\documentclass[12pt]{article}
\usepackage{amsmath,amssymb,amsthm}
\numberwithin{equation}{section}
\usepackage{hyperref}
\usepackage{units}
\usepackage{color}
\usepackage[T1]{fontenc}
\usepackage[utf8]{inputenc}
\usepackage{authblk}
\usepackage{bm}
\usepackage{graphicx}

\usepackage{datetime}

\usepackage[toc,page]{appendix}

\newcommand{\Z}{{\mathbb Z}}

\newtheorem{thm}{Theorem}

\newtheorem{lemma}{Lemma}

\newtheorem{cor}[thm]{Corollary}

\title{Identities for second order recurrence sequences}

\author[]{Kunle Adegoke \thanks{adegoke00@gmail.com}}

\affil{Department of Physics and Engineering Physics, \mbox{Obafemi Awolowo University}, 220005 Ile-Ife, Nigeria}

\begin{document}

\date{}

\maketitle

\begin{abstract}
\noindent We derive several identities for arbitrary homogeneous second order recurrence sequences with constant coefficients. The results are then applied to present a harmonized study of six well known integer sequences, namely the Fibonacci sequence, the sequence of Lucas numbers, the Jacobsthal sequence, the Jacobsthal-Lucas sequence, the Pell sequence and the Pell-Lucas sequence.
\end{abstract}
\tableofcontents
\section{Introduction}
Our aim in writing this paper is to derive several identities for arbitrary second order recurrence sequences with constant coefficients. As a concrete illustration of how our results may be put to use, we will derive identities for the integer sequences mentioned in the abstract and defined below.

\bigskip

The Fibonacci numbers, $F_n$, and the Lucas numbers, $L_n$, are defined, for \mbox{$n\in\Z$}, as usual, through the recurrence relations \mbox{$F_n=F_{n-1}+F_{n-2}$ ($n\ge 2$)}, \mbox{$F_0=0$, $F_1=1$} and \mbox{$L_n=L_{n-1}+L_{n-2}$} ($n\ge 2$), $L_0=2$, $L_1=1$, with $F_{-n}=(-1)^{n-1}F_n$ and $L_{-n}=(-1)^nL_n$. Exhaustive discussion of the properties of Fibonacci and Lucas numbers can be found in Vajda~\cite{vajda} and in Koshy~\cite{koshy}.

\bigskip

The Jacobsthal numbers, $J_n$, and the Jacobsthal-Lucas numbers, $j_n$, are defined, for \mbox{$n\in\Z$}, through the recurrence relations \mbox{$J_n=J_{n-1}+2J_{n-2}$} ($n\ge 2$), \mbox{$J_0=0$, $J_1=1$} and \mbox{$j_n=j_{n-1}+2j_{n-2}$} ($n\ge 2$), $j_0=2$, $j_1=1$, with $J_{-n}=(-1)^{n-1}2^{-n}J_n$ and $j_{-n}=(-1)^n2^{-n}j_n$. Horadam~\cite{horadam96} and Aydin~\cite{aydin} are good reference materials on the Jacobsthal and associated sequences. 

\bigskip

The Pell numbers, $P_n$, and Pell-Lucas numbers, $Q_n$, are defined, for \mbox{$n\in\Z$}, through the recurrence relations \mbox{$P_n=2P_{n-1}+P_{n-2}$} ($n\ge 2$), \mbox{$P_0=0$, $P_1=1$} and \mbox{$Q_n=2Q_{n-1}+Q_{n-2}$} ($n\ge 2$), $Q_0=2$, $Q_1=1$, with $P_{-n}=(-1)^{n-1}P_n$ and $Q_{-n}=(-1)^nQ_n$. Koshy~\cite{koshy14}, Horadam~\cite{horadam71} and Patel and Shrivastava~\cite{patel13} are useful source materials on Pell and Pell-Lucas numbers.

\bigskip

Note that, in this paper, apart from in the binomial summation identities where the upper limit must be non-negative, the upper limit in the summation identities is allowed to take on negative values once we adopt the summation convention that, if $k<0$ then
\[
\sum_{r=0}^k{f_r}\equiv-\sum_{r=k+1}^{-1}{f_r}\,,
\]
as long as $f_r$ is not singular in the summation interval.

\bigskip

Here is a couple of results to whet the reader's appetite for reading on:

\bigskip

From Corollary~\ref{cor.o3h0k66}:
\[
F_{n + h} L_{n + k}  - F_n L_{n + h + k}  = ( - 1)^n F_h L_k\,,
\]
\[
J_{n + h} j_{n + k}  - J_n j_{n + h + k}  = ( - 1)^n 2^n J_h j_k\,,
\]
\[
P_{n + h} Q_{n + k}  - P_n Q_{n + h + k}  = ( - 1)^n P_h Q_k\,.
\]
From Theorem~\ref{thm.x3r8tfg}:
\[
( - 1)^u L_u^2  + ( - 1)^v L_v^2  + ( - 1)^w L_w^2  = ( - 1)^w L_u L_v L_w  + 4\,,
\]
\[
( - 1)^u 2^v j_u^2  + ( - 1)^v 2^u j_v^2  + ( - 1)^w j_w^2  = ( - 1)^w j_u j_v j_w  + 2^{w + 2}
\]
and
\[
( - 1)^u Q_u^2  + ( - 1)^v Q_v^2  + ( - 1)^w Q_w^2  = ( - 1)^w Q_u Q_v Q_w  + 4\,,
\]
for integers $u$, $v$, $w$ such that $u+v=w$.

\bigskip

From Theorem~\ref{thm.d7ll8n8}, for nonnegative integer $k$ and arbitrary integers $a$, $b$, $c$, $d$, $e$, $m$ for which the denominator does not vanish:
\[
\begin{split}
&\sum\limits_{r = 0}^k {\binom kr\left( {\frac{{F_{d - c} F_{e - b}  - F_{e - c} F_{d - b}}}{{F_{d - a} F_{e - c}  - F_{e - a} F_{d - c}}}} \right)^r F_{m - (b-c) k + (b-a)r} }\\
&\qquad = \left( {\frac{{F_{d - a} F_{e - b}  - F_{e - a} F_{d - b}}}{{F_{d - a} F_{e - c}  - F_{e - a} F_{d - c}}}} \right)^k F_m\,,
\end{split}
\]
\[
\begin{split}
&\sum\limits_{r = 0}^k {\binom kr\left( {\frac{{J_{d - c} J_{e - b}  - J_{e - c} J_{d - b}}}{{J_{d - a} J_{e - c}  - J_{e - a} J_{d - c}}}} \right)^r J_{m - (b-c) k + (b-a)r} }\\
&\qquad = \left( {\frac{{J_{d - a} J_{e - b}  - J_{e - a} J_{d - b}}}{{J_{d - a} J_{e - c}  - J_{e - a} J_{d - c}}}} \right)^k J_m\,,
\end{split}
\]
\[
\begin{split}
&\sum\limits_{r = 0}^k {\binom kr\left( {\frac{{P_{d - c} P_{e - b}  - P_{e - c} P_{d - b}}}{{P_{d - a} P_{e - c}  - P_{e - a} P_{d - c}}}} \right)^r P_{m - (b-c) k + (b-a)r} }\\
&\qquad = \left( {\frac{{P_{d - a} P_{e - b}  - P_{e - a} P_{d - b}}}{{P_{d - a} P_{e - c}  - P_{e - a} P_{d - c}}}} \right)^k P_m\,.
\end{split}
\]
\section{Main results}
\subsection{Identities}
\begin{lemma}\label{lem.xs3wlq8}
Let $\{X_m\}$ and $\{Y_m\}$, $m\in\Z$, be homogeneous second order recurrence sequences with constant coefficients. Let $\{X_m\}$ and $\{Y_m\}$ possess the same recurrence relation. Let $\Delta_{xy} = X_{d - a} Y_{e - b}  - X_{e - a} Y_{d - b}$. Then, the identity:
\[
\begin{split}
&(X_{d - a} Y_{e - b}  - X_{e - a} Y_{d - b} )X_{m - c}\\
&\quad= (X_{d - c} Y_{e - b}  - X_{e - c} Y_{d - b} )X_{m - a}\\ 
&\qquad+ (X_{d - a} X_{e - c}  - X_{e - a} X_{d - c} )Y_{m - b}\,, 
\end{split}
\]
holds for arbitrary integers $a$, $b$, $c$, $d$, $e$ and $m$ for which $\Delta_{xy}\ne 0$.
\end{lemma}
\begin{proof}
By hypothesis, $\{X_m\}$ and $\{Y_m\}$ have the same recurrence relations, therefore we seek a relation of the following type:
\begin{equation}\label{eq.dx4yk9j}
X_{m-c}=\lambda_1X_{m-a}+\lambda_2Y_{m-b}\,,
\end{equation}
between any three numbers $X_{m-c}$, $X_{m-a}$ and $Y_{m-b}$, where $a$, $b$ and $c$ are fixed integers and $\lambda_1$ and $\lambda_2$ are suitable constants.
Evaluating~\eqref{eq.dx4yk9j} at $m=d$ and at $m=e$ produces two equations:
\begin{equation}\label{eq.fvstnur}
X_{d - c}  = \lambda _1 X_{d - a}  + \lambda _2 Y_{d - b}
\end{equation}
and
\begin{equation}\label{eq.haxa0oy}
X_{e - c}  = \lambda _1 X_{e - a}  + \lambda _2 Y_{e - b}\,,
\end{equation}
to be solved simultaneously for the constants $\lambda_1$ and $\lambda_2$.
Solutions exist if
\[
\Delta _{xy}  = \left| {\begin{array}{*{20}c}
   {X_{d - a} } & {Y_{d - b} }  \\
   {X_{e - a} } & {Y_{e - b} }  \\
\end{array}} \right| = X_{d - a} Y_{e - b}  - X_{e - a}Y_{d - b}  \ne 0\,.
\]

The result follows from substituting into~\eqref{eq.dx4yk9j} the $\lambda_1$ and $\lambda_2$ found from solving~\eqref{eq.fvstnur} and~\eqref{eq.haxa0oy}.

\end{proof}
\begin{lemma}\label{lem.ggc97nq}
Let $\{X_m\}$, $m\in\Z$, be a homogeneous second order recurrence sequence with constant coefficients. Then, the following identity holds for arbitrary integers $a$, $b$, $c$, $d$, $e$ and $m$:
\[
\begin{split}
&(X_{d - a} X_{e - b}  - X_{e - a} X_{d - b} )X_{m - c}\\
&\quad= (X_{d - c} X_{e - b}  - X_{e - c} X_{d - b} )X_{m - a}\\ 
&\qquad+ (X_{d - a} X_{e - c}  - X_{e - a} X_{d - c} )X_{m - b}\,. 
\end{split}
\]

\end{lemma}
\begin{proof}
Let $\Delta_{xx}=X_{d - a} X_{e - b}  - X_{e - a} X_{d - b}\ne 0$ and proceed as in the proof of Lemma~\ref{lem.xs3wlq8}. We have
\begin{equation}\label{eq.rdg1fww}
\begin{split}
&(X_{d - a} X_{e - b}  - X_{e - a} X_{d - b} )X_{m - c}\\
&\quad= (X_{d - c} X_{e - b}  - X_{e - c} X_{d - b} )X_{m - a}\\ 
&\qquad+ (X_{d - a} X_{e - c}  - X_{e - a} X_{d - c} )X_{m - b}\,. 
\end{split}
\end{equation}
But we will now prove that the identity~\eqref{eq.rdg1fww} continues to hold even if $\Delta_{xx}=0$. Let
\[
\Delta_1=X_{d - c} X_{e - b}  - X_{e - c} X_{d - b},\quad\Delta_2=X_{d - a} X_{e - c}  - X_{e - a} X_{d - c}\,.
\]
There are six possible situations in which $\Delta_{xx}$ can vanish. We consider them in turn.
\begin{enumerate}
\item[1.] $X_{d-a}=0=X_{e-a}$, in which case $d=e\Rightarrow$ $\Delta_1=\Delta_2=0$ and hence identity~\eqref{eq.rdg1fww} remains valid. 
\item[2.] $X_{e-b}=0=X_{d-b}$, in which case, again,  $d=e\Rightarrow$ $\Delta_1=\Delta_2=0$ and hence identity~\eqref{eq.rdg1fww} remains valid
\item[3.] $X_{d-a}=0=X_{d-b}$, in which case $b=a$ and the right side of identity~\eqref{eq.rdg1fww} reads
\[
(X_{d - c} X_{e - a}  - X_{e - c} X_{d - a} )X_{m - a}+ (X_{d - a} X_{e - c}  - X_{e - a} X_{d - c} )X_{m - a}\,,
\]
which evaluates to zero, so that identity~\eqref{eq.rdg1fww} remains valid.
\item[4.]  $X_{e-b}=0=X_{d-b}$, in which case $e=d$ and the right side of identity~\eqref{eq.rdg1fww} reads 
\[
(X_{d - c} X_{d - b}  - X_{d - c} X_{d - b} )X_{m - a} + (X_{d - a} X_{d - c}  - X_{d - a} X_{d - c} )X_{m - b}\,,
\]
which evaluates to zero, so that identity~\eqref{eq.rdg1fww} remains valid.
\item[5.] $X_{d-a}=X_{e-a}$ and $X_{e-b}=X_{d-b}$, in which case, again,  $d=e\Rightarrow$ $\Delta_1=\Delta_2=0$ and hence identity~\eqref{eq.rdg1fww} remains valid.
\item[6.] $X_{d-a}=X_{d-b}$ and $X_{e-b}=X_{e-a}$, in which case, as in case~3,  $b=a$ and hence identity~\eqref{eq.rdg1fww} remains valid.

\end{enumerate}
Thus we see that identity~\eqref{eq.rdg1fww} is valid regardless of the nature of $\Delta_{xx}$, so that the identity holds for all integers.

\end{proof}
\begin{lemma}\label{lem.r460krb}
Let $\{X_m\}$, $m\in\Z$, be a homogeneous second order recurrence sequence with constant coefficients. Then, the following identity holds for arbitrary integers $a$, $b$, $c$ and $m$:
\[
\begin{split}
&(X_0{}^2  - X_{b- a} X_{a- b} )X_{m - c}\\
&\quad= (X_{a- c} X_0  - X_{b- c} X_{a- b} )X_{m - a}\\ 
&\qquad+ (X_0 X_{b- c}  - X_{b- a} X_{a- c} )X_{m - b}\,. 
\end{split}
\]

\end{lemma}
\subsection{Summation identities}
The following identities are obtained by making appropriate substitutions from Lemmata~\ref{lem.xs3wlq8}~and~\ref{lem.ggc97nq} into Lemmata~1 and~2 of~\cite{adegoke18}.
\begin{lemma}\label{lem.mbs9sdf}
Let $\{X_m\}$ and $\{Y_m\}$, $m\in\Z$, be homogeneous second order recurrence sequences with constant coefficients. Let $\{X_m\}$ and $\{Y_m\}$ possess the same recurrence relation. Let $\Delta_{xy} = X_{d - a} Y_{e - b}  - X_{e - a} Y_{d - b}$, $\Delta_{1}=X_{d - c} Y_{e - b}  - X_{e - c} Y_{d - b}$ and $\Delta_{2}=X_{d - a} X_{e - c}  - X_{e - a} X_{d - c}$. Then, the following identity holds for arbitrary integers $a$, $b$, $c$, $d$, $e$, $m$ and $k$ for which $\Delta_{xy}\ne 0$, $\Delta_{1}\ne 0$, $\Delta_{2}\ne 0$:
\begin{equation}
\begin{split}
&\sum\limits_{r = 0}^k {\left( {\frac{{X_{d - a} Y_{e - b}  - X_{e - a} Y_{d - b}}}{{X_{d - c} Y_{e - b}  - X_{e - c} Y_{d - b}}}} \right)^r Y_{m - k(a-c)  - b + c  + (a-c) r} }\\
&\qquad = \left( {\frac{{X_{d - a} Y_{e - b}  - X_{e - a} Y_{d - b}}}{{X_{d - a} X_{e - c}  - X_{e - a} X_{d - c}}}} \right)\left( {\frac{{X_{d - a} Y_{e - b}  - X_{e - a} Y_{d - b}}}{{X_{d - c} Y_{e - b}  - X_{e - c} Y_{d - b}}}} \right)^k X_m\\
&\quad\qquad- \left( {\frac{{X_{d - c} Y_{e - b}  - X_{e - c} Y_{d - b}}}{{X_{d - a} X_{e - c}  - X_{e - a} X_{d - c}}}} \right)X_{m - (k + 1)(a-c) }\,.
\end{split}
\end{equation}

\end{lemma}
\begin{lemma}\label{lem.ubk4var}
Let $\{X_m\}$, $m\in\Z$, be a homogeneous second order recurrence sequence with constant coefficients. Let $\Delta_{1}=X_{d - c} X_{e - b}  - X_{e - c} X_{d - b}$ and $\Delta_{2}=X_{d - a} X_{e - c}  - X_{e - a} X_{d - c}$. Then, the following identities hold for integer $k$ and arbitrary integers $a$, $b$, $c$, $d$, $e$ and $m$ for which $\Delta_{1}\ne 0$ and $\Delta_{2}\ne 0$:
\begin{equation}
\begin{split}
&\sum\limits_{r = 0}^k {\left( {\frac{{X_{d - a} X_{e - b}  - X_{e - a} X_{d - b}}}{{X_{d - c} X_{e - b}  - X_{e - c} X_{d - b}}}} \right)^r X_{m - k(a-c)  - b + c  + (a-c) r} }\\
&\qquad = \left( {\frac{{X_{d - a} X_{e - b}  - X_{e - a} X_{d - b}}}{{X_{d - a} X_{e - c}  - X_{e - a} X_{d - c}}}} \right)\left( {\frac{{X_{d - a} X_{e - b}  - X_{e - a} X_{d - b}}}{{X_{d - c} X_{e - b}  - X_{e - c} X_{d - b}}}} \right)^k X_m\\
&\quad\qquad- \left( {\frac{{X_{d - c} X_{e - b}  - X_{e - c} X_{d - b}}}{{X_{d - a} X_{e - c}  - X_{e - a} X_{d - c}}}} \right)X_{m - (k + 1)(a-c) }\,,
\end{split}
\end{equation}

\begin{equation}
\begin{split}
&\sum\limits_{r = 0}^k {\left( {\frac{{X_{d - a} X_{e - b}  - X_{e - a} X_{d - b}}}{{X_{d - a} X_{e - c}  - X_{e - a} X_{d - c}}}} \right)^r X_{m - k(b-c)  - a + c  + (b-c) r} }\\
&\qquad = \left( {\frac{{X_{d - a} X_{e - b}  - X_{e - a} X_{d - b}}}{{X_{d - c} X_{e - b}  - X_{e - c} X_{d - b}}}} \right)\left( {\frac{{X_{d - a} X_{e - b}  - X_{e - a} X_{d - b}}}{{X_{d - a} X_{e - c}  - X_{e - a} X_{d - c}}}} \right)^k X_m\\
&\quad\qquad- \left( {\frac{{X_{d - a} X_{e - c}  - X_{e - a} X_{d - c}}}{{X_{d - c} X_{e - b}  - X_{e - c} X_{d - b}}}} \right)X_{m - (k + 1)(b-c) }
\end{split}
\end{equation}
and
\begin{equation}
\begin{split}
&\sum\limits_{r = 0}^k {\left( {\frac{{X_{e - a} X_{d - c}-X_{d - a} X_{e - c}}}{{X_{d - c} X_{e - b}  - X_{e - c} X_{d - b}}}} \right)^r X_{m - k(a-b)  + b - c  + (a-b) r} }\\
&\qquad = \left( {\frac{{X_{d - a} X_{e - c}  - X_{e - a} X_{d - c}}}{{X_{d - a} X_{e - b}  - X_{e - a} X_{d - b}}}} \right)\left( {\frac{{X_{e - a} X_{d - c}-X_{d - a} X_{e - c}}}{{X_{d - c} X_{e - b}  - X_{e - c} X_{d - b}}}} \right)^k X_m\\
&\quad\qquad+ \left( {\frac{X_{d - c} X_{e - b}  - X_{e - c} X_{d - b}}{X_{d - a} X_{e - b}  - X_{e - a} X_{d - b}}} \right)X_{m - (k + 1)(a-b) }\,.
\end{split}
\end{equation}

\end{lemma}

\subsection{Binomial summation identities}
The following identities are obtained by making appropriate substitutions from the identity of Lemma~\ref{lem.ggc97nq} into the identities of Lemma~3 of~\cite{adegoke18}.
\begin{lemma}\label{lem.fg6x899}
Let $\{X_m\}$, $m\in\Z$, be a homogeneous second order recurrence sequence with constant coefficients. Let $\Delta_{1}=X_{d - c} X_{e - b}  - X_{e - c} X_{d - b}$ and $\Delta_{2}=X_{d - a} X_{e - c}  - X_{e - a} X_{d - c}$. Then, the following identity holds for positive integer $k$ and arbitrary integers $a$, $b$, $c$, $d$, $e$ and $m$ for which $\Delta_{1}\ne 0$ and $\Delta_{2}\ne 0$:
\begin{equation}
\begin{split}
&\sum\limits_{r = 0}^k {\binom kr\left( {\frac{{X_{d - c} X_{e - b}  - X_{e - c} X_{d - b}}}{{X_{d - a} X_{e - c}  - X_{e - a} X_{d - c}}}} \right)^r X_{m - (b-c) k + (b-a)r} }\\
&\qquad = \left( {\frac{{X_{d - a} X_{e - b}  - X_{e - a} X_{d - b}}}{{X_{d - a} X_{e - c}  - X_{e - a} X_{d - c}}}} \right)^k X_m\,,
\end{split}
\end{equation}
\begin{equation}
\begin{split}
&\sum\limits_{r = 0}^k {\binom kr\left( {\frac{{X_{e - a} X_{d - b} - X_{d - a} X_{e - b}}}{{X_{d - a} X_{e - c}  - X_{e - a} X_{d - c}}}} \right)^r X_{m + (a-b)k + (b-c) r} }\\
&\qquad = \left( {\frac{{X_{d - c} X_{e - b}  - X_{e - c} X_{d - b}}}{{X_{e - a} X_{d - c} - X_{d - a} X_{e - c}}}} \right)^k X_m
\end{split}
\end{equation}
and
\begin{equation}
\begin{split}
&\sum\limits_{r = 0}^k {\binom kr\left( {\frac{{X_{e - a} X_{d - b} - X_{d - a} X_{e - b}}}{{X_{d - c} X_{e - b}  - X_{e - c} X_{d - b}}}} \right)^r X_{m + (b-a)k + (a-c) r} }\\
&\qquad = \left( {\frac{{X_{d - a} X_{e - c}  - X_{e - a} X_{d - c}}}{{X_{e - c} X_{d - b } - X_{d - c} X_{e - b}}}} \right)^k X_m\,.
\end{split}
\end{equation}

\end{lemma}

\section{Applications and examples}
We now employ the results of the previous section to give a combined study of six well known integer sequences. First we give a modified version of Lemma~\ref{lem.xs3wlq8} that allows the removal of the $\Delta_{xy}$ condition.
\begin{lemma}\label{lem.k8oc7py}
Let $\{X_m\}$ and $\{Y_m\}$, $m\in\Z$, be homogeneous second order recurrence sequences with constant coefficients. Let $\{X_m\}$ and $\{Y_m\}$ possess the same recurrence relation. Let $Y_m\ne 0$ for all integers $m$. Finally, let $\{X_m\}$ and $\{Y_m\}$ have at most three members in common. Then, the identity:
\[
\begin{split}
&(X_{d - a} Y_{e - b}  - X_{e - a} Y_{d - b} )X_{m - c}\\
&\quad= (X_{d - c} Y_{e - b}  - X_{e - c} Y_{d - b} )X_{m - a}\\ 
&\qquad+ (X_{d - a} X_{e - c}  - X_{e - a} X_{d - c} )Y_{m - b}\,, 
\end{split}
\]
holds for arbitrary integers $a$, $b$, $c$, $d$, $e$ and $m$.
\end{lemma}
\begin{proof}
Let $\Delta_{xy} = X_{d - a} Y_{e - b}  - X_{e - a} Y_{d - b}$. According to Lemma~\ref{lem.xs3wlq8} we have
\begin{equation}\label{eq.sfgh6vg}
\begin{split}
&(X_{d - a} Y_{e - b}  - X_{e - a} Y_{d - b} )X_{m - c}\\
&\quad= (X_{d - c} Y_{e - b}  - X_{e - c} Y_{d - b} )X_{m - a}\\ 
&\qquad+ (X_{d - a} X_{e - c}  - X_{e - a} X_{d - c} )Y_{m - b}\,, 
\end{split}
\end{equation}
provided that $\Delta_{xy}\ne 0$. But we will now prove that the identity~\eqref{eq.sfgh6vg} continues to hold even if $\Delta_{xy}=0$. Let
\[
\Delta_1=X_{d - c} Y_{e - b}  - X_{e - c} Y_{d - b},\quad\Delta_2=X_{d - a} X_{e - c}  - X_{e - a} X_{d - c}\,.
\]
$\Delta_{xy}$ vanishes under the following conditions: 
\begin{enumerate}
\item[1.] $X_{d-a}=0=X_{e-a}$, in which case $d=e\Rightarrow$ $\Delta_1=\Delta_2=0$ and hence identity~\eqref{eq.sfgh6vg} remains valid. 
\item[2.] $X_{d-a}=X_{e-a}$ and $Y_{e-b}=Y_{d-b}$, in which case, again, $d=e\Rightarrow$ $\Delta_1=\Delta_2=0$ and hence identity~\eqref{eq.sfgh6vg} remains valid.
\end{enumerate}
Thus we see that identity~\eqref{eq.sfgh6vg} is valid regardless of the nature of $\Delta_{xy}$, so that the identity holds for all integers.

\end{proof}
\subsection{Identities}
Our first set of results comes from choosing an appropriate $(X,Y)$ pair, in each case, from the set $\{F,L,J,j,P,Q\}$ and using it in Lemma~\ref{lem.k8oc7py}.
\begin{thm}\label{thm.oi8sled}
The following identities hold for arbitrary integers $a$, $b$, $c$, $d$, $e$ and $m$:
\begin{equation}\label{eq.otygaej}
\begin{split}
&(F_{d - a} L_{e - b}  - F_{e - a} L_{d - b} )F_{m - c}\\
&\quad= (F_{d - c} L_{e - b}  - F_{e - c} L_{d - b} )F_{m - a}\\ 
&\qquad+ (F_{d - a} F_{e - c}  - F_{e - a} F_{d - c} )L_{m - b}\,, 
\end{split}
\end{equation}
\begin{equation}
\begin{split}
&(L_{d - a} F_{e - b}  - L_{e - a} F_{d - b} )L_{m - c}\\
&\quad= (L_{d - c} F_{e - b}  - L_{e - c} F_{d - b} )L_{m - a}\\ 
&\qquad+ (L_{d - a} L_{e - c}  - L_{e - a} L_{d - c} )F_{m - b}\,, 
\end{split}
\end{equation}
\begin{equation}\label{eq.l7fcm6e}
\begin{split}
&(J_{d - a} j_{e - b}  - J_{e - a} j_{d - b} )J_{m - c}\\
&\quad= (J_{d - c} j_{e - b}  - J_{e - c} j_{d - b} )J_{m - a}\\ 
&\qquad+ (J_{d - a} J_{e - c}  - J_{e - a} J_{d - c} )j_{m - b}\,, 
\end{split}
\end{equation}
\begin{equation}
\begin{split}
&(j_{d - a} J_{e - b}  - j_{e - a} J_{d - b} )j_{m - c}\\
&\quad= (j_{d - c} J_{e - b}  - j_{e - c} J_{d - b} )j_{m - a}\\ 
&\qquad+ (j_{d - a} j_{e - c}  - j_{e - a} j_{d - c} )J_{m - b}\,, 
\end{split}
\end{equation}
\begin{equation}\label{eq.ld9puri}
\begin{split}
&(P_{d - a} Q_{e - b}  - P_{e - a} Q_{d - b} )P_{m - c}\\
&\quad= (P_{d - c} Q_{e - b}  - P_{e - c} Q_{d - b} )P_{m - a}\\ 
&\qquad+ (P_{d - a} P_{e - c}  - P_{e - a} P_{d - c} )Q_{m - b} 
\end{split}
\end{equation}
and
\begin{equation}
\begin{split}
&(Q_{d - a} P_{e - b}  - Q_{e - a} P_{d - b} )Q_{m - c}\\
&\quad= (Q_{d - c} P_{e - b}  - Q_{e - c} P_{d - b} )Q_{m - a}\\ 
&\qquad+ (Q_{d - a} Q_{e - c}  - Q_{e - a} Q_{d - c} )P_{m - b}\,. 
\end{split}
\end{equation}

\end{thm}
To demonstrate how known identities may be recovered (and further new ones discovered), set $m=c$ in identities~\eqref{eq.otygaej},~\eqref{eq.l7fcm6e} and~\eqref{eq.ld9puri} of Theorem~\ref{thm.oi8sled} to obtain the following result.
\begin{cor}\label{cor.o3h0k66}
The following identities hold for integers $a$, $b$, $c$, $d$ and $e$:
\begin{equation}
(F_{d - c} L_{e - b}  - F_{e - c} L_{d - b} )F_{c - a}  = (F_{e - a} F_{d - c}  - F_{d - a} F_{e - c} )L_{c - b}\,,
\end{equation}
\begin{equation}
(J_{d - c} j_{e - b}  - J_{e - c} j_{d - b} )J_{c - a}  = (J_{e - a} J_{d - c}  - J_{d - a} J_{e - c} )j_{c - b}
\end{equation}
and
\begin{equation}
(P_{d - c} Q_{e - b}  - P_{e - c} Q_{d - b} )P_{c - a}  = (P_{e - a} P_{d - c}  - P_{d - a} P_{e - c} )Q_{c - b}\,.
\end{equation}

\end{cor}
Upon setting $e=a$ in the identities of Corollary~\ref{cor.o3h0k66} and using $F_{a-c}=(-1)^{a-c-1}F_{c-a}$, $J_{a-c}=(-1)^{a-c-1}2^{a-c}J_{c-a}$ and $P_{a-c}=(-1)^{a-c-1}P_{c-a}$ we obtain
\begin{equation}\label{eq.di2otyl}
F_{d - c} L_{a - b}  - F_{a - c} L_{d - b}  = ( - 1)^{a - c} F_{d - a} L_{c - b}\,,
\end{equation}
\begin{equation}\label{eq.l0w62ax}
J_{d - c} j_{a - b}  - J_{a - c} j_{d - b}  = ( - 1)^{a - c}2^{a-c} J_{d - a} j_{c - b}
\end{equation}
and
\begin{equation}\label{eq.a41el9k}
P_{d - c} Q_{a - b}  - P_{a - c} Q_{d - b}  = ( - 1)^{a - c} P_{d - a} Q_{c - b}\,.
\end{equation}
In order to write the above identities using three parameters, we set $a=d-h$, $b=d-n-h-k$ and $c=d-n-h$, obtaining
\begin{equation}
F_{n + h} L_{n + k}  - F_n L_{n + h + k}  = ( - 1)^n F_h L_k\,,
\end{equation}
\begin{equation}
J_{n + h} j_{n + k}  - J_n j_{n + h + k}  = ( - 1)^n 2^n J_h j_k
\end{equation}
and
\begin{equation}
P_{n + h} Q_{n + k}  - P_n Q_{n + h + k}  = ( - 1)^n P_h Q_k\,.
\end{equation}
Setting $c=b$ in the identities~\eqref{eq.di2otyl}, \eqref{eq.l0w62ax} and \eqref{eq.a41el9k}, we have
\begin{equation}\label{eq.shlj0oj}
F_{d - b} L_{a - b}  - F_{a - b} L_{d - b}  = ( - 1)^{a - b} 2F_{d - a} \,,
\end{equation}
\begin{equation}\label{eq.y4uq86a}
J_{d - b} j_{a - b}  - J_{a - b} j_{d - b}  = ( - 1)^{a - b}2^{a-b+1} J_{d - a}
\end{equation}
and
\begin{equation}\label{eq.hb4ciww}
P_{d - b} Q_{a - b}  - P_{a - b} Q_{d - b}  = ( - 1)^{a - b} 2P_{d - a}\,.
\end{equation}
Two parameter forms are obtained by setting $d-b=u$ and $a-b=v$, giving
\begin{equation}
F_u L_v  - F_v L_u  = ( - 1)^v 2F_{u - v}\,,
\end{equation}
\begin{equation}
J_u j_v  - J_v j_u  = ( - 1)^v 2^{v + 1} J_{u - v}
\end{equation}
and
\begin{equation}
P_u Q_v  - P_v Q_u  = ( - 1)^v 2P_{u - v}\,.
\end{equation}
Setting $b=0$, $c=-a$ in identities~\eqref{eq.di2otyl}, \eqref{eq.l0w62ax} and \eqref{eq.a41el9k} gives
\begin{equation}
F_{d + a}  - ( - 1)^a F_{d - a}  = F_a L_d\,,
\end{equation}
\begin{equation}
J_{d + a}  - ( - 1)^a 2^a J_{d - a}  = J_a j_d
\end{equation}
and
\begin{equation}
P_{d + a}  - ( - 1)^a P_{d - a}  = P_a Q_d\,.
\end{equation}
Choosing $b=c=0$, $e=a+d$ in the identities in Corollary~\ref{cor.o3h0k66} and making use of the identities~\eqref{eq.shlj0oj}, \eqref{eq.y4uq86a} and \eqref{eq.hb4ciww}, we obtain Catalan's identities:
\begin{equation}
F_d{}^2-F_{d-a}F_{d+a}=(-1)^{d-a}F_a{}^2\,,
\end{equation}
\begin{equation}\label{eq.ilkwci7}
J_d{}^2-J_{d-a}J_{d+a}=(-1)^{d-a}2^{d-a}J_a{}^2
\end{equation}
and
\begin{equation}\label{eq.nfgkvo5}
P_d{}^2-P_{d-a}P_{d+a}=(-1)^{d-a}P_a{}^2\,.
\end{equation}
Note that identity~(2.23) of Horadam~\cite{horadam96} is a special case of identity~\eqref{eq.ilkwci7} while identity~(30) of~\cite{horadam71} is a special case of~\eqref{eq.nfgkvo5}.

\bigskip

Upon setting $d=0$, $c=-a$ in Corollary~\ref{cor.o3h0k66} and making use of identities~\eqref{eq.huqy5k3}, \eqref{eq.qizngjv} and \eqref{eq.huqy5k3}, we obtain:
\begin{equation}\label{eq.niyqtne}
F_e L_{a + b}  + ( - 1)^b F_a L_{e - b}  = F_{e + a} L_b\,,
\end{equation}
\begin{equation}\label{eq.ec8fvvs}
J_e j_{a + b}  + ( - 1)^b2^b J_a j_{e - b}  = J_{e + a} j_b
\end{equation}
and
\begin{equation}\label{eq.ijp5aoh}
P_e Q_{a + b}  + ( - 1)^b P_a Q_{e - b}  = P_{e + a} Q_b\,.
\end{equation}
Puttig $e=a$ in~\eqref{eq.niyqtne} --- \eqref{eq.ijp5aoh} produces
\begin{equation}\label{eq.cysiwix}
L_{a + b}  + ( - 1)^b L_{a - b}  = L_a L_b
\end{equation}
\begin{equation}\label{eq.rort88d}
j_{a + b}  + ( - 1)^b 2^bj_{a - b}  = j_a j_b
\end{equation}
\begin{equation}\label{eq.cysiwix}
Q_{a + b}  + ( - 1)^b Q_{a - b}  = Q_a L_b\,,
\end{equation}
while using $b=0$ in the identities gives
\begin{equation}\label{eq.orz0o21}
F_eL_a+F_aL_e=2F_{e+a}\,,
\end{equation}
\begin{equation}\label{eq.ezz5hsl}
J_ej_a+J_aj_e=2J_{e+a}
\end{equation}
and
\begin{equation}\label{eq.rbpn1n6}
P_eQ_a+P_aQ_e=2P_{e+a}\,.
\end{equation}
Finally, setting $e=b$ in the same identities~\eqref{eq.niyqtne} --- \eqref{eq.ijp5aoh} gives
\begin{equation}\label{eq.db4gr6o}
F_{a+b}L_b-F_bL_{a+b}=(-1)^b2F_a\,,
\end{equation}
\begin{equation}\label{eq.gue3kys}
J_{a+b}j_b-J_bj_{a+b}=(-1)^b2^{b+1}J_a
\end{equation}
and
\begin{equation}\label{eq.ocqzd20}
P_{a+b}Q_b-P_bQ_{a+b}=(-1)^b2P_a\,.
\end{equation}
The choice $e=2u+b$, $a=b$, $d=b$ and $c=b+u$ in Corollary~\ref{cor.o3h0k66} yields the identities:
\begin{equation}\label{eq.qgjcd5f}
L_{2u}+(-1)^u2=L_u{}^2\,,
\end{equation}
\begin{equation}\label{eq.n89h5sa}
j_{2u}+(-1)^u2^{u+1}=j_u{}^2
\end{equation}
and
\begin{equation}\label{eq.vcrqvxj}
Q_{2u}+(-1)^u2=Q_u{}^2\,.
\end{equation}
Note that identities~\eqref{eq.qgjcd5f}, \eqref{eq.n89h5sa} and \eqref{eq.vcrqvxj} can also be obtained directly from identities~\eqref{eq.di2otyl}, \eqref{eq.l0w62ax} and \eqref{eq.a41el9k} by setting $b=a$, $c=a+u$ and $d=a+2u$.

\bigskip

Lemma~\ref{lem.ggc97nq} invites the following results.
\begin{thm}
The following identities hold for integers $a$, $b$, $c$, $d$, $e$ and $m$:
\begin{equation}\label{eq.seuftvx}
\begin{split}
&(F_{d - a} F_{e - b}  - F_{e - a} F_{d - b} )F_{m - c}\\
&\quad= (F_{d - c} F_{e - b}  - F_{e - c} F_{d - b} )F_{m - a}\\ 
&\qquad+ (F_{d - a} F_{e - c}  - F_{e - a} F_{d - c} )F_{m - b}\,, 
\end{split}
\end{equation}
\begin{equation}
\begin{split}
&(L_{d - a} L_{e - b}  - L_{e - a} L_{d - b} )L_{m - c}\\
&\quad= (L_{d - c} L_{e - b}  - L_{e - c} L_{d - b} )L_{m - a}\\ 
&\qquad+ (L_{d - a} L_{e - c}  - L_{e - a} L_{d - c} )L_{m - b}\,, 
\end{split}
\end{equation}
\begin{equation}\label{eq.f9cwnj1}
\begin{split}
&(J_{d - a} J_{e - b}  - J_{e - a} J_{d - b} )J_{m - c}\\
&\quad= (J_{d - c} J_{e - b}  - J_{e - c} J_{d - b} )J_{m - a}\\ 
&\qquad+ (J_{d - a} J_{e - c}  - J_{e - a} J_{d - c} )J_{m - b}\,, 
\end{split}
\end{equation}
\begin{equation}
\begin{split}
&(j_{d - a} j_{e - b}  - j_{e - a} j_{d - b} )j_{m - c}\\
&\quad= (j_{d - c} j_{e - b}  - j_{e - c} j_{d - b} )j_{m - a}\\ 
&\qquad+ (j_{d - a} j_{e - c}  - j_{e - a} j_{d - c} )j_{m - b}\,, 
\end{split}
\end{equation}
\begin{equation}\label{eq.tkovapp}
\begin{split}
&(P_{d - a} P_{e - b}  - P_{e - a} P_{d - b} )P_{m - c}\\
&\quad= (P_{d - c} P_{e - b}  - P_{e - c} P_{d - b} )P_{m - a}\\ 
&\qquad+ (P_{d - a} P_{e - c}  - P_{e - a} P_{d - c} )P_{m - b}
\end{split}
\end{equation}
and
\begin{equation}
\begin{split}
&(Q_{d - a} Q_{e - b}  - Q_{e - a} Q_{d - b} )Q_{m - c}\\
&\quad= (Q_{d - c} Q_{e - b}  - Q_{e - c} Q_{d - b} )Q_{m - a}\\ 
&\qquad+ (Q_{d - a} Q_{e - c}  - Q_{e - a} Q_{d - c} )Q_{m - b}\,. 
\end{split}
\end{equation}

\end{thm}
Setting $m=b$ in identities~\eqref{eq.seuftvx}, \eqref{eq.f9cwnj1} and \eqref{eq.tkovapp} gives the next set of results.
\begin{cor}\label{thm.phyv4vj}
The following identities hold for integers $a$, $b$, $c$, $d$ and $e$:
\begin{equation}\label{eq.lbf5tex}
\begin{split}
&(F_{d - a} F_{e - b}  - F_{e - a} F_{d - b} )F_{b - c}= (F_{d - c} F_{e - b}  - F_{e - c} F_{d - b} )F_{b - a}\,, 
\end{split}
\end{equation}
\begin{equation}\label{eq.j3e1sjn}
\begin{split}
&(J_{d - a} J_{e - b}  - J_{e - a} J_{d - b} )J_{b - c}= (J_{d - c} J_{e - b}  - J_{e - c} J_{d - b} )J_{b - a} 
\end{split}
\end{equation}
and
\begin{equation}\label{eq.hci7jcu}
\begin{split}
&(P_{d - a} P_{e - b}  - P_{e - a} P_{d - b} )P_{b - c}= (P_{d - c} P_{e - b}  - P_{e - c} P_{d - b} )P_{b - a}\,.
\end{split}
\end{equation}

\end{cor}
Using $b=0$, $c=-a$ and $e=a$ in the identities in Corollary~\ref{thm.phyv4vj}, we obtain:
\begin{equation}\label{eq.huqy5k3}
F_{d + a}  + ( - 1)^a F_{d - a}  = F_d L_a\,,
\end{equation}
\begin{equation}\label{eq.qizngjv}
J_{d + a}  + ( - 1)^a 2^a J_{d - a}  = J_d j_a
\end{equation}
and
\begin{equation}\label{eq.fjuoodj}
P_{d + a}  + ( - 1)^a P_{d - a}  = P_d Q_a\,.
\end{equation}

Putting $d=a$ in each case, the identities in Corollary~\ref{thm.phyv4vj} reduce to
\begin{equation}\label{eq.kjtr5i5}
F_{a - c} F_{e - b}  - F_{e - c} F_{a - b}  = ( - 1)^{a - b} F_{e - a} F_{b - c}\,,
\end{equation}
\begin{equation}\label{eq.ht7q6bl}
J_{a - c} J_{e - b}  - J_{e - c} J_{a - b}  = ( - 1)^{a - b} 2^{a-b}J_{e - a} J_{b - c}
\end{equation}
and
\begin{equation}\label{eq.ml4n5lc}
P_{a - c} P_{e - b}  - P_{e - c} P_{a - b}  = ( - 1)^{a - b} P_{e - a} P_{b - c}\,.
\end{equation}
Using $a=e+h$, $b=e-n-k$, $c=e-n$, identities~\eqref{eq.kjtr5i5} --- \eqref{eq.ml4n5lc} can also be written
\begin{equation}\label{eq.dt415os}
F_{n+h} F_{n+k}  - F_{n} F_{n+h+k}  = ( - 1)^{n} F_{h} F_{k}\,,
\end{equation}
\begin{equation}\label{eq.wsgwwty}
J_{n+h} J_{n+k}  - J_{n} J_{n+h+k}  = ( - 1)^{n} 2^nJ_{h} J_{k}
\end{equation}
and
\begin{equation}\label{eq.qqylvxr}
P_{n+h} P_{n+k}  - P_{n} P_{n+h+k}  = ( - 1)^{n} P_{h} P_{k}\,.
\end{equation}
\begin{thm}\label{thm.x3r8tfg}
The following identities hold for all integers $a$, $b$ and $c$:
\begin{equation}
\begin{split}
&( - 1)^{a - b} L_{a - b}^2  + ( - 1)^{b - c} L_{b - c}^2  + ( - 1)^{a - c} L_{a - c}^2\\
&\qquad  = ( - 1)^{a - c} L_{a - b}L_{b - c} L_{a - c}   + 4\,,
\end{split}
\end{equation}
\begin{equation}
\begin{split}
&( - 1)^{a - b} 2^{b - a} j_{a - b}^2  + ( - 1)^{b - c} 2^{c - b} j_{b - c}^2  + ( - 1)^{a - c} 2^{c - a} j_{a - c}^2\\
&\qquad  = ( - 1)^{a - c} 2^{c - a} j_{a - b} j_{b - c}j_{a - c}   + 4
\end{split}
\end{equation}
and
\begin{equation}
\begin{split}
&( - 1)^{a - b} Q_{a - b}^2  + ( - 1)^{b - c} Q_{b - c}^2  + ( - 1)^{a - c} Q_{a - c}^2\\
&\qquad  = ( - 1)^{a - c} Q_{a - b}Q_{b - c} Q_{a - c}   + 4\,.
\end{split}
\end{equation}

\end{thm}
\begin{proof}
Set $m=c$ in Lemma~\ref{lem.r460krb} and use $X=L$, $X=j$ and $X=Q$, in turn.
\end{proof}
Note that, for integers $u$, $v$, $w$ such that $u+v=w$, the identities in Theorem~\ref{thm.x3r8tfg} can also be written
\begin{equation}
( - 1)^u L_u^2  + ( - 1)^v L_v^2  + ( - 1)^w L_w^2  = ( - 1)^w L_u L_v L_w  + 4\,,
\end{equation}
\begin{equation}
( - 1)^u 2^v j_u^2  + ( - 1)^v 2^u j_v^2  + ( - 1)^w j_w^2  = ( - 1)^w j_u j_v j_w  + 2^{w + 2}
\end{equation}
and
\begin{equation}
( - 1)^u Q_u^2  + ( - 1)^v Q_v^2  + ( - 1)^w Q_w^2  = ( - 1)^w Q_u Q_v Q_w  + 4\,.
\end{equation}
\subsection{Weighted sums}
Choosing an appropriate $(X,Y)$ pair, in each case, from the set $\{F,L,J,j,P,Q\}$ and using it in Lemma~\ref{lem.mbs9sdf} we have the next set of results.
\begin{thm}
The following identities hold for any integer $k$ and arbitrary integers $a$, $b$, $c$, $d$, $e$, $m$ for which the denominator does not vanish:
\begin{equation}
\begin{split}
&\sum\limits_{r = 0}^k {\left( {\frac{{F_{d - a} L_{e - b}  - F_{e - a} L_{d - b}}}{{F_{d - c} L_{e - b}  - F_{e - c} L_{d - b}}}} \right)^r L_{m - k(a-c)  - b + c  + (a-c) r} }\\
&\qquad = \left( {\frac{{F_{d - a} L_{e - b}  - F_{e - a} L_{d - b}}}{{F_{d - a} F_{e - c}  - F_{e - a} F_{d - c}}}} \right)\left( {\frac{{F_{d - a} L_{e - b}  - F_{e - a} L_{d - b}}}{{F_{d - c} L_{e - b}  - F_{e - c} L_{d - b}}}} \right)^k F_m\\
&\qquad\qquad\qquad- \left( {\frac{{F_{d - c} L_{e - b}  - F_{e - c} L_{d - b}}}{{F_{d - a} F_{e - c}  - F_{e - a} F_{d - c}}}} \right)F_{m - (k + 1)(a-c) }\,,
\end{split}
\end{equation}
\begin{equation}
\begin{split}
&\sum\limits_{r = 0}^k {\left( {\frac{{L_{d - a} F_{e - b}  - L_{e - a} F_{d - b}}}{{L_{d - c} F_{e - b}  - L_{e - c} F_{d - b}}}} \right)^r F_{m - k(a-c)  - b + c  + (a-c) r} }\\
&\qquad = \left( {\frac{{L_{d - a} F_{e - b}  - L_{e - a} F_{d - b}}}{{L_{d - a} L_{e - c}  - L_{e - a} L_{d - c}}}} \right)\left( {\frac{{L_{d - a} F_{e - b}  - L_{e - a} F_{d - b}}}{{L_{d - c} F_{e - b}  - L_{e - c} F_{d - b}}}} \right)^k L_m\\
&\qquad\qquad\qquad- \left( {\frac{{L_{d - c} F_{e - b}  - L_{e - c} F_{d - b}}}{{L_{d - a} L_{e - c}  - L_{e - a} L_{d - c}}}} \right)L_{m - (k + 1)(a-c) }\,,
\end{split}
\end{equation}
\begin{equation}
\begin{split}
&\sum\limits_{r = 0}^k {\left( {\frac{{J_{d - a} j_{e - b}  - J_{e - a} j_{d - b}}}{{J_{d - c} j_{e - b}  - J_{e - c} j_{d - b}}}} \right)^r j_{m - k(a-c)  - b + c  + (a-c) r} }\\
&\qquad = \left( {\frac{{J_{d - a} j_{e - b}  - J_{e - a} j_{d - b}}}{{J_{d - a} J_{e - c}  - J_{e - a} J_{d - c}}}} \right)\left( {\frac{{J_{d - a} j_{e - b}  - J_{e - a} j_{d - b}}}{{J_{d - c} j_{e - b}  - J_{e - c} j_{d - b}}}} \right)^k J_m\\
&\qquad\qquad\qquad- \left( {\frac{{J_{d - c} j_{e - b}  - J_{e - c} j_{d - b}}}{{J_{d - a} J_{e - c}  - J_{e - a} J_{d - c}}}} \right)J_{m - (k + 1)(a-c) }\,,
\end{split}
\end{equation}
\begin{equation}
\begin{split}
&\sum\limits_{r = 0}^k {\left( {\frac{{j_{d - a} J_{e - b}  - j_{e - a} J_{d - b}}}{{j_{d - c} J_{e - b}  - j_{e - c} J_{d - b}}}} \right)^r J_{m - k(a-c)  - b + c  + (a-c) r} }\\
&\qquad = \left( {\frac{{j_{d - a} J_{e - b}  - j_{e - a} J_{d - b}}}{{j_{d - a} j_{e - c}  - j_{e - a} j_{d - c}}}} \right)\left( {\frac{{j_{d - a} J_{e - b}  - j_{e - a} J_{d - b}}}{{j_{d - c} J_{e - b}  - j_{e - c} J_{d - b}}}} \right)^k j_m\\
&\qquad\qquad\qquad- \left( {\frac{{j_{d - c} J_{e - b}  - j_{e - c} J_{d - b}}}{{j_{d - a} j_{e - c}  - j_{e - a} j_{d - c}}}} \right)j_{m - (k + 1)(a-c) }\,,
\end{split}
\end{equation}
\begin{equation}
\begin{split}
&\sum\limits_{r = 0}^k {\left( {\frac{{P_{d - a} Q_{e - b}  - P_{e - a} Q_{d - b}}}{{P_{d - c} Q_{e - b}  - P_{e - c} Q_{d - b}}}} \right)^r Q_{m - k(a-c)  - b + c  + (a-c) r} }\\
&\qquad = \left( {\frac{{P_{d - a} Q_{e - b}  - P_{e - a} Q_{d - b}}}{{P_{d - a} P_{e - c}  - P_{e - a} P_{d - c}}}} \right)\left( {\frac{{P_{d - a} Q_{e - b}  - P_{e - a} Q_{d - b}}}{{P_{d - c} Q_{e - b}  - P_{e - c} Q_{d - b}}}} \right)^k P_m\\
&\qquad\qquad\qquad- \left( {\frac{{P_{d - c} Q_{e - b}  - P_{e - c} Q_{d - b}}}{{P_{d - a} P_{e - c}  - P_{e - a} P_{d - c}}}} \right)P_{m - (k + 1)(a-c) }
\end{split}
\end{equation}
and
\begin{equation}
\begin{split}
&\sum\limits_{r = 0}^k {\left( {\frac{{Q_{d - a} P_{e - b}  - Q_{e - a} P_{d - b}}}{{Q_{d - c} P_{e - b}  - Q_{e - c} P_{d - b}}}} \right)^r P_{m - k(a-c)  - b + c  + (a-c) r} }\\
&\qquad = \left( {\frac{{Q_{d - a} P_{e - b}  - Q_{e - a} P_{d - b}}}{{Q_{d - a} Q_{e - c}  - Q_{e - a} Q_{d - c}}}} \right)\left( {\frac{{Q_{d - a} P_{e - b}  - Q_{e - a} P_{d - b}}}{{Q_{d - c} P_{e - b}  - Q_{e - c} P_{d - b}}}} \right)^k Q_m\\
&\qquad\qquad\qquad- \left( {\frac{{Q_{d - c} P_{e - b}  - Q_{e - c} P_{d - b}}}{{Q_{d - a} Q_{e - c}  - Q_{e - a} Q_{d - c}}}} \right)Q_{m - (k + 1)(a-c) }\,.
\end{split}
\end{equation}

\end{thm}
Using $X=F$, $X=L$, $X=J$, $X=j$, $X=P$, $X=Q$, in turn, in Lemma~\ref{lem.ubk4var} gives the next results.
\begin{thm}
The following identities hold for any integer $k$ and arbitrary integers $a$, $b$, $c$, $d$, $e$, $m$ for which the denominator does not vanish:
\begin{equation}
\begin{split}
&\sum\limits_{r = 0}^k {\left( {\frac{{F_{d - a} F_{e - b}  - F_{e - a} F_{d - b}}}{{F_{d - c} F_{e - b}  - F_{e - c} F_{d - b}}}} \right)^r F_{m - k(a-c)  - b + c  + (a-c) r} }\\
&\qquad = \left( {\frac{{F_{d - a} F_{e - b}  - F_{e - a} F_{d - b}}}{{F_{d - a} F_{e - c}  - F_{e - a} F_{d - c}}}} \right)\left( {\frac{{F_{d - a} F_{e - b}  - F_{e - a} F_{d - b}}}{{F_{d - c} F_{e - b}  - F_{e - c} F_{d - b}}}} \right)^k F_m\\
&\quad\qquad- \left( {\frac{{F_{d - c} F_{e - b}  - F_{e - c} F_{d - b}}}{{F_{d - a} F_{e - c}  - F_{e - a} F_{d - c}}}} \right)F_{m - (k + 1)(a-c) }\,,
\end{split}
\end{equation}

\begin{equation}
\begin{split}
&\sum\limits_{r = 0}^k {\left( {\frac{{F_{d - a} F_{e - b}  - F_{e - a} F_{d - b}}}{{F_{d - a} F_{e - c}  - F_{e - a} F_{d - c}}}} \right)^r F_{m - k(b-c)  - a + c  + (b-c) r} }\\
&\qquad = \left( {\frac{{F_{d - a} F_{e - b}  - F_{e - a} F_{d - b}}}{{F_{d - c} F_{e - b}  - F_{e - c} F_{d - b}}}} \right)\left( {\frac{{F_{d - a} F_{e - b}  - F_{e - a} F_{d - b}}}{{F_{d - a} F_{e - c}  - F_{e - a} F_{d - c}}}} \right)^k F_m\\
&\quad\qquad- \left( {\frac{{F_{d - a} F_{e - c}  - F_{e - a} F_{d - c}}}{{F_{d - c} F_{e - b}  - F_{e - c} F_{d - b}}}} \right)F_{m - (k + 1)(b-c) }\,,
\end{split}
\end{equation}
\begin{equation}
\begin{split}
&\sum\limits_{r = 0}^k {\left( {\frac{{F_{e - a} F_{d - c}-F_{d - a} F_{e - c}}}{{F_{d - c} F_{e - b}  - F_{e - c} F_{d - b}}}} \right)^r F_{m - k(a-b)  + b - c  + (a-b) r} }\\
&\qquad = \left( {\frac{{F_{d - a} F_{e - c}  - F_{e - a} F_{d - c}}}{{F_{d - a} F_{e - b}  - F_{e - a} F_{d - b}}}} \right)\left( {\frac{{F_{e - a} F_{d - c}-F_{d - a} F_{e - c}}}{{F_{d - c} F_{e - b}  - F_{e - c} F_{d - b}}}} \right)^k F_m\\
&\quad\qquad+ \left( {\frac{F_{d - c} F_{e - b}  - F_{e - c} F_{d - b}}{F_{d - a} F_{e - b}  - F_{e - a} F_{d - b}}} \right)F_{m - (k + 1)(a-b) }\,,
\end{split}
\end{equation}
\begin{equation}
\begin{split}
&\sum\limits_{r = 0}^k {\left( {\frac{{L_{d - a} L_{e - b}  - L_{e - a} L_{d - b}}}{{L_{d - c} L_{e - b}  - L_{e - c} L_{d - b}}}} \right)^r L_{m - k(a-c)  - b + c  + (a-c) r} }\\
&\qquad = \left( {\frac{{L_{d - a} L_{e - b}  - L_{e - a} L_{d - b}}}{{L_{d - a} L_{e - c}  - L_{e - a} L_{d - c}}}} \right)\left( {\frac{{L_{d - a} L_{e - b}  - L_{e - a} L_{d - b}}}{{L_{d - c} L_{e - b}  - L_{e - c} L_{d - b}}}} \right)^k L_m\\
&\quad\qquad- \left( {\frac{{L_{d - c} L_{e - b}  - L_{e - c} L_{d - b}}}{{L_{d - a} L_{e - c}  - L_{e - a} L_{d - c}}}} \right)L_{m - (k + 1)(a-c) }\,,
\end{split}
\end{equation}

\begin{equation}
\begin{split}
&\sum\limits_{r = 0}^k {\left( {\frac{{L_{d - a} L_{e - b}  - L_{e - a} L_{d - b}}}{{L_{d - a} L_{e - c}  - L_{e - a} L_{d - c}}}} \right)^r L_{m - k(b-c)  - a + c  + (b-c) r} }\\
&\qquad = \left( {\frac{{L_{d - a} L_{e - b}  - L_{e - a} L_{d - b}}}{{L_{d - c} L_{e - b}  - L_{e - c} L_{d - b}}}} \right)\left( {\frac{{L_{d - a} L_{e - b}  - L_{e - a} L_{d - b}}}{{L_{d - a} L_{e - c}  - L_{e - a} L_{d - c}}}} \right)^k L_m\\
&\quad\qquad- \left( {\frac{{L_{d - a} L_{e - c}  - L_{e - a} L_{d - c}}}{{L_{d - c} L_{e - b}  - L_{e - c} L_{d - b}}}} \right)L_{m - (k + 1)(b-c) }\,,
\end{split}
\end{equation}
\begin{equation}
\begin{split}
&\sum\limits_{r = 0}^k {\left( {\frac{{L_{e - a} L_{d - c}-L_{d - a} L_{e - c}}}{{L_{d - c} L_{e - b}  - L_{e - c} L_{d - b}}}} \right)^r L_{m - k(a-b)  + b - c  + (a-b) r} }\\
&\qquad = \left( {\frac{{L_{d - a} L_{e - c}  - L_{e - a} L_{d - c}}}{{L_{d - a} L_{e - b}  - L_{e - a} L_{d - b}}}} \right)\left( {\frac{{L_{e - a} L_{d - c}-L_{d - a} L_{e - c}}}{{L_{d - c} L_{e - b}  - L_{e - c} L_{d - b}}}} \right)^k L_m\\
&\quad\qquad+ \left( {\frac{L_{d - c} L_{e - b}  - L_{e - c} L_{d - b}}{L_{d - a} L_{e - b}  - L_{e - a} L_{d - b}}} \right)L_{m - (k + 1)(a-b) }\,,
\end{split}
\end{equation}
\begin{equation}
\begin{split}
&\sum\limits_{r = 0}^k {\left( {\frac{{J_{d - a} J_{e - b}  - J_{e - a} J_{d - b}}}{{J_{d - c} J_{e - b}  - J_{e - c} J_{d - b}}}} \right)^r J_{m - k(a-c)  - b + c  + (a-c) r} }\\
&\qquad = \left( {\frac{{J_{d - a} J_{e - b}  - J_{e - a} J_{d - b}}}{{J_{d - a} J_{e - c}  - J_{e - a} J_{d - c}}}} \right)\left( {\frac{{J_{d - a} J_{e - b}  - J_{e - a} J_{d - b}}}{{J_{d - c} J_{e - b}  - J_{e - c} J_{d - b}}}} \right)^k J_m\\
&\quad\qquad- \left( {\frac{{J_{d - c} J_{e - b}  - J_{e - c} J_{d - b}}}{{J_{d - a} J_{e - c}  - J_{e - a} J_{d - c}}}} \right)J_{m - (k + 1)(a-c) }\,,
\end{split}
\end{equation}

\begin{equation}
\begin{split}
&\sum\limits_{r = 0}^k {\left( {\frac{{J_{d - a} J_{e - b}  - J_{e - a} J_{d - b}}}{{J_{d - a} J_{e - c}  - J_{e - a} J_{d - c}}}} \right)^r J_{m - k(b-c)  - a + c  + (b-c) r} }\\
&\qquad = \left( {\frac{{J_{d - a} J_{e - b}  - J_{e - a} J_{d - b}}}{{J_{d - c} J_{e - b}  - J_{e - c} J_{d - b}}}} \right)\left( {\frac{{J_{d - a} J_{e - b}  - J_{e - a} J_{d - b}}}{{J_{d - a} J_{e - c}  - J_{e - a} J_{d - c}}}} \right)^k J_m\\
&\quad\qquad- \left( {\frac{{J_{d - a} J_{e - c}  - J_{e - a} J_{d - c}}}{{J_{d - c} J_{e - b}  - J_{e - c} J_{d - b}}}} \right)J_{m - (k + 1)(b-c) }\,,
\end{split}
\end{equation}
\begin{equation}
\begin{split}
&\sum\limits_{r = 0}^k {\left( {\frac{{J_{e - a} J_{d - c}-J_{d - a} J_{e - c}}}{{J_{d - c} J_{e - b}  - J_{e - c} J_{d - b}}}} \right)^r J_{m - k(a-b)  + b - c  + (a-b) r} }\\
&\qquad = \left( {\frac{{J_{d - a} J_{e - c}  - J_{e - a} J_{d - c}}}{{J_{d - a} J_{e - b}  - J_{e - a} J_{d - b}}}} \right)\left( {\frac{{J_{e - a} J_{d - c}-J_{d - a} J_{e - c}}}{{J_{d - c} J_{e - b}  - J_{e - c} J_{d - b}}}} \right)^k J_m\\
&\quad\qquad+ \left( {\frac{J_{d - c} J_{e - b}  - J_{e - c} J_{d - b}}{J_{d - a} J_{e - b}  - J_{e - a} J_{d - b}}} \right)J_{m - (k + 1)(a-b) }\,,
\end{split}
\end{equation}
\begin{equation}
\begin{split}
&\sum\limits_{r = 0}^k {\left( {\frac{{j_{d - a} j_{e - b}  - j_{e - a} j_{d - b}}}{{j_{d - c} j_{e - b}  - j_{e - c} j_{d - b}}}} \right)^r j_{m - k(a-c)  - b + c  + (a-c) r} }\\
&\qquad = \left( {\frac{{j_{d - a} j_{e - b}  - j_{e - a} j_{d - b}}}{{j_{d - a} j_{e - c}  - j_{e - a} j_{d - c}}}} \right)\left( {\frac{{j_{d - a} j_{e - b}  - j_{e - a} j_{d - b}}}{{j_{d - c} j_{e - b}  - j_{e - c} j_{d - b}}}} \right)^k j_m\\
&\quad\qquad- \left( {\frac{{j_{d - c} j_{e - b}  - j_{e - c} j_{d - b}}}{{j_{d - a} j_{e - c}  - j_{e - a} j_{d - c}}}} \right)j_{m - (k + 1)(a-c) }\,,
\end{split}
\end{equation}

\begin{equation}
\begin{split}
&\sum\limits_{r = 0}^k {\left( {\frac{{j_{d - a} j_{e - b}  - j_{e - a} j_{d - b}}}{{j_{d - a} j_{e - c}  - j_{e - a} j_{d - c}}}} \right)^r j_{m - k(b-c)  - a + c  + (b-c) r} }\\
&\qquad = \left( {\frac{{j_{d - a} j_{e - b}  - j_{e - a} j_{d - b}}}{{j_{d - c} j_{e - b}  - j_{e - c} j_{d - b}}}} \right)\left( {\frac{{j_{d - a} j_{e - b}  - j_{e - a} j_{d - b}}}{{j_{d - a} j_{e - c}  - j_{e - a} j_{d - c}}}} \right)^k j_m\\
&\quad\qquad- \left( {\frac{{j_{d - a} j_{e - c}  - j_{e - a} j_{d - c}}}{{j_{d - c} j_{e - b}  - j_{e - c} j_{d - b}}}} \right)j_{m - (k + 1)(b-c) }\,,
\end{split}
\end{equation}
\begin{equation}
\begin{split}
&\sum\limits_{r = 0}^k {\left( {\frac{{j_{e - a} j_{d - c}-j_{d - a} j_{e - c}}}{{j_{d - c} j_{e - b}  - j_{e - c} j_{d - b}}}} \right)^r j_{m - k(a-b)  + b - c  + (a-b) r} }\\
&\qquad = \left( {\frac{{j_{d - a} j_{e - c}  - j_{e - a} j_{d - c}}}{{j_{d - a} j_{e - b}  - j_{e - a} j_{d - b}}}} \right)\left( {\frac{{j_{e - a} j_{d - c}-j_{d - a} j_{e - c}}}{{j_{d - c} j_{e - b}  - j_{e - c} j_{d - b}}}} \right)^k j_m\\
&\quad\qquad+ \left( {\frac{j_{d - c} j_{e - b}  - j_{e - c} j_{d - b}}{j_{d - a} j_{e - b}  - j_{e - a} j_{d - b}}} \right)j_{m - (k + 1)(a-b) }\,,
\end{split}
\end{equation}
\begin{equation}
\begin{split}
&\sum\limits_{r = 0}^k {\left( {\frac{{P_{d - a} P_{e - b}  - P_{e - a} P_{d - b}}}{{P_{d - c} P_{e - b}  - P_{e - c} P_{d - b}}}} \right)^r P_{m - k(a-c)  - b + c  + (a-c) r} }\\
&\qquad = \left( {\frac{{P_{d - a} P_{e - b}  - P_{e - a} P_{d - b}}}{{P_{d - a} P_{e - c}  - P_{e - a} P_{d - c}}}} \right)\left( {\frac{{P_{d - a} P_{e - b}  - P_{e - a} P_{d - b}}}{{P_{d - c} P_{e - b}  - P_{e - c} P_{d - b}}}} \right)^k P_m\\
&\quad\qquad- \left( {\frac{{P_{d - c} P_{e - b}  - P_{e - c} P_{d - b}}}{{P_{d - a} P_{e - c}  - P_{e - a} P_{d - c}}}} \right)P_{m - (k + 1)(a-c) }\,,
\end{split}
\end{equation}

\begin{equation}
\begin{split}
&\sum\limits_{r = 0}^k {\left( {\frac{{P_{d - a} P_{e - b}  - P_{e - a} P_{d - b}}}{{P_{d - a} P_{e - c}  - P_{e - a} P_{d - c}}}} \right)^r P_{m - k(b-c)  - a + c  + (b-c) r} }\\
&\qquad = \left( {\frac{{P_{d - a} P_{e - b}  - P_{e - a} P_{d - b}}}{{P_{d - c} P_{e - b}  - P_{e - c} P_{d - b}}}} \right)\left( {\frac{{P_{d - a} P_{e - b}  - P_{e - a} P_{d - b}}}{{P_{d - a} P_{e - c}  - P_{e - a} P_{d - c}}}} \right)^k P_m\\
&\quad\qquad- \left( {\frac{{P_{d - a} P_{e - c}  - P_{e - a} P_{d - c}}}{{P_{d - c} P_{e - b}  - P_{e - c} P_{d - b}}}} \right)P_{m - (k + 1)(b-c) }\,,
\end{split}
\end{equation}
\begin{equation}
\begin{split}
&\sum\limits_{r = 0}^k {\left( {\frac{{P_{e - a} P_{d - c}-P_{d - a} P_{e - c}}}{{P_{d - c} P_{e - b}  - P_{e - c} P_{d - b}}}} \right)^r P_{m - k(a-b)  + b - c  + (a-b) r} }\\
&\qquad = \left( {\frac{{P_{d - a} P_{e - c}  - P_{e - a} P_{d - c}}}{{P_{d - a} P_{e - b}  - P_{e - a} P_{d - b}}}} \right)\left( {\frac{{P_{e - a} P_{d - c}-P_{d - a} P_{e - c}}}{{P_{d - c} P_{e - b}  - P_{e - c} P_{d - b}}}} \right)^k P_m\\
&\quad\qquad+ \left( {\frac{P_{d - c} P_{e - b}  - P_{e - c} P_{d - b}}{P_{d - a} P_{e - b}  - P_{e - a} P_{d - b}}} \right)P_{m - (k + 1)(a-b) }\,,
\end{split}
\end{equation}
\begin{equation}
\begin{split}
&\sum\limits_{r = 0}^k {\left( {\frac{{Q_{d - a} Q_{e - b}  - Q_{e - a} Q_{d - b}}}{{Q_{d - c} Q_{e - b}  - Q_{e - c} Q_{d - b}}}} \right)^r Q_{m - k(a-c)  - b + c  + (a-c) r} }\\
&\qquad = \left( {\frac{{Q_{d - a} Q_{e - b}  - Q_{e - a} Q_{d - b}}}{{Q_{d - a} Q_{e - c}  - Q_{e - a} Q_{d - c}}}} \right)\left( {\frac{{Q_{d - a} Q_{e - b}  - Q_{e - a} Q_{d - b}}}{{Q_{d - c} Q_{e - b}  - Q_{e - c} Q_{d - b}}}} \right)^k Q_m\\
&\quad\qquad- \left( {\frac{{Q_{d - c} Q_{e - b}  - Q_{e - c} Q_{d - b}}}{{Q_{d - a} Q_{e - c}  - Q_{e - a} Q_{d - c}}}} \right)Q_{m - (k + 1)(a-c) }\,,
\end{split}
\end{equation}

\begin{equation}
\begin{split}
&\sum\limits_{r = 0}^k {\left( {\frac{{Q_{d - a} Q_{e - b}  - Q_{e - a} Q_{d - b}}}{{Q_{d - a} Q_{e - c}  - Q_{e - a} Q_{d - c}}}} \right)^r Q_{m - k(b-c)  - a + c  + (b-c) r} }\\
&\qquad = \left( {\frac{{Q_{d - a} Q_{e - b}  - Q_{e - a} Q_{d - b}}}{{Q_{d - c} Q_{e - b}  - Q_{e - c} Q_{d - b}}}} \right)\left( {\frac{{Q_{d - a} Q_{e - b}  - Q_{e - a} Q_{d - b}}}{{Q_{d - a} Q_{e - c}  - Q_{e - a} Q_{d - c}}}} \right)^k Q_m\\
&\quad\qquad- \left( {\frac{{Q_{d - a} Q_{e - c}  - Q_{e - a} Q_{d - c}}}{{Q_{d - c} Q_{e - b}  - Q_{e - c} Q_{d - b}}}} \right)Q_{m - (k + 1)(b-c) }
\end{split}
\end{equation}
and
\begin{equation}
\begin{split}
&\sum\limits_{r = 0}^k {\left( {\frac{{Q_{e - a} Q_{d - c}-Q_{d - a} Q_{e - c}}}{{Q_{d - c} Q_{e - b}  - Q_{e - c} Q_{d - b}}}} \right)^r Q_{m - k(a-b)  + b - c  + (a-b) r} }\\
&\qquad = \left( {\frac{{Q_{d - a} Q_{e - c}  - Q_{e - a} Q_{d - c}}}{{Q_{d - a} Q_{e - b}  - Q_{e - a} Q_{d - b}}}} \right)\left( {\frac{{Q_{e - a} Q_{d - c}-Q_{d - a} Q_{e - c}}}{{Q_{d - c} Q_{e - b}  - Q_{e - c} Q_{d - b}}}} \right)^k Q_m\\
&\quad\qquad+ \left( {\frac{Q_{d - c} Q_{e - b}  - Q_{e - c} Q_{d - b}}{Q_{d - a} Q_{e - b}  - Q_{e - a} Q_{d - b}}} \right)Q_{m - (k + 1)(a-b) }\,.
\end{split}
\end{equation}

\end{thm}
\subsection{Weighted binomial sums}
Using $X=F$, $X=L$, $X=J$, $X=j$, $X=P$, $X=Q$, in turn, in Lemma~\ref{lem.fg6x899} gives the next results.
\begin{thm}\label{thm.d7ll8n8}
The following identities hold for nonnegative integer $k$ and arbitrary integers $a$, $b$, $c$, $d$, $e$, $m$ for which the denominator does not vanish:
\begin{equation}
\begin{split}
&\sum\limits_{r = 0}^k {\binom kr\left( {\frac{{F_{d - c} F_{e - b}  - F_{e - c} F_{d - b}}}{{F_{d - a} F_{e - c}  - F_{e - a} F_{d - c}}}} \right)^r F_{m - (b-c) k + (b-a)r} }\\
&\qquad = \left( {\frac{{F_{d - a} F_{e - b}  - F_{e - a} F_{d - b}}}{{F_{d - a} F_{e - c}  - F_{e - a} F_{d - c}}}} \right)^k F_m\,,
\end{split}
\end{equation}
\begin{equation}
\begin{split}
&\sum\limits_{r = 0}^k {\binom kr\left( {\frac{{F_{e - a} F_{d - b} - F_{d - a} F_{e - b}}}{{F_{d - a} F_{e - c}  - F_{e - a} F_{d - c}}}} \right)^r F_{m + (a-b)k + (b-c) r} }\\
&\qquad = \left( {\frac{{F_{d - c} F_{e - b}  - F_{e - c} F_{d - b}}}{{F_{e - a} F_{d - c} - F_{d - a} F_{e - c}}}} \right)^k F_m\,,
\end{split}
\end{equation}
\begin{equation}
\begin{split}
&\sum\limits_{r = 0}^k {\binom kr\left( {\frac{{F_{e - a} F_{d - b} - F_{d - a} F_{e - b}}}{{F_{d - c} F_{e - b}  - F_{e - c} F_{d - b}}}} \right)^r F_{m + (b-a)k + (a-c) r} }\\
&\qquad = \left( {\frac{{F_{d - a} F_{e - c}  - F_{e - a} F_{d - c}}}{{F_{e - c} F_{d - b } - F_{d - c} F_{e - b}}}} \right)^k F_m\,,
\end{split}
\end{equation}
\begin{equation}
\begin{split}
&\sum\limits_{r = 0}^k {\binom kr\left( {\frac{{L_{d - c} L_{e - b}  - L_{e - c} L_{d - b}}}{{L_{d - a} L_{e - c}  - L_{e - a} L_{d - c}}}} \right)^r L_{m - (b-c) k + (b-a)r} }\\
&\qquad = \left( {\frac{{L_{d - a} L_{e - b}  - L_{e - a} L_{d - b}}}{{L_{d - a} L_{e - c}  - L_{e - a} L_{d - c}}}} \right)^k L_m\,,
\end{split}
\end{equation}
\begin{equation}
\begin{split}
&\sum\limits_{r = 0}^k {\binom kr\left( {\frac{{L_{e - a} L_{d - b} - L_{d - a} L_{e - b}}}{{L_{d - a} L_{e - c}  - L_{e - a} L_{d - c}}}} \right)^r L_{m + (a-b)k + (b-c) r} }\\
&\qquad = \left( {\frac{{L_{d - c} L_{e - b}  - L_{e - c} L_{d - b}}}{{L_{e - a} L_{d - c} - L_{d - a} L_{e - c}}}} \right)^k L_m\,,
\end{split}
\end{equation}
\begin{equation}
\begin{split}
&\sum\limits_{r = 0}^k {\binom kr\left( {\frac{{L_{e - a} L_{d - b} - L_{d - a} L_{e - b}}}{{L_{d - c} L_{e - b}  - L_{e - c} L_{d - b}}}} \right)^r L_{m + (b-a)k + (a-c) r} }\\
&\qquad = \left( {\frac{{L_{d - a} L_{e - c}  - L_{e - a} L_{d - c}}}{{L_{e - c} L_{d - b } - L_{d - c} L_{e - b}}}} \right)^k L_m\,,
\end{split}
\end{equation}
\begin{equation}
\begin{split}
&\sum\limits_{r = 0}^k {\binom kr\left( {\frac{{J_{d - c} J_{e - b}  - J_{e - c} J_{d - b}}}{{J_{d - a} J_{e - c}  - J_{e - a} J_{d - c}}}} \right)^r J_{m - (b-c) k + (b-a)r} }\\
&\qquad = \left( {\frac{{J_{d - a} J_{e - b}  - J_{e - a} J_{d - b}}}{{J_{d - a} J_{e - c}  - J_{e - a} J_{d - c}}}} \right)^k J_m\,,
\end{split}
\end{equation}
\begin{equation}
\begin{split}
&\sum\limits_{r = 0}^k {\binom kr\left( {\frac{{J_{e - a} J_{d - b} - J_{d - a} J_{e - b}}}{{J_{d - a} J_{e - c}  - J_{e - a} J_{d - c}}}} \right)^r J_{m + (a-b)k + (b-c) r} }\\
&\qquad = \left( {\frac{{J_{d - c} J_{e - b}  - J_{e - c} J_{d - b}}}{{J_{e - a} J_{d - c} - J_{d - a} J_{e - c}}}} \right)^k J_m\,,
\end{split}
\end{equation}
\begin{equation}
\begin{split}
&\sum\limits_{r = 0}^k {\binom kr\left( {\frac{{J_{e - a} J_{d - b} - J_{d - a} J_{e - b}}}{{J_{d - c} J_{e - b}  - J_{e - c} J_{d - b}}}} \right)^r J_{m + (b-a)k + (a-c) r} }\\
&\qquad = \left( {\frac{{J_{d - a} J_{e - c}  - J_{e - a} J_{d - c}}}{{J_{e - c} J_{d - b } - J_{d - c} J_{e - b}}}} \right)^k J_m\,,
\end{split}
\end{equation}
\begin{equation}
\begin{split}
&\sum\limits_{r = 0}^k {\binom kr\left( {\frac{{j_{d - c} j_{e - b}  - j_{e - c} j_{d - b}}}{{j_{d - a} j_{e - c}  - j_{e - a} j_{d - c}}}} \right)^r j_{m - (b-c) k + (b-a)r} }\\
&\qquad = \left( {\frac{{j_{d - a} j_{e - b}  - j_{e - a} j_{d - b}}}{{j_{d - a} j_{e - c}  - j_{e - a} j_{d - c}}}} \right)^k j_m\,,
\end{split}
\end{equation}
\begin{equation}
\begin{split}
&\sum\limits_{r = 0}^k {\binom kr\left( {\frac{{j_{e - a} j_{d - b} - j_{d - a} j_{e - b}}}{{j_{d - a} j_{e - c}  - j_{e - a} j_{d - c}}}} \right)^r j_{m + (a-b)k + (b-c) r} }\\
&\qquad = \left( {\frac{{j_{d - c} j_{e - b}  - j_{e - c} j_{d - b}}}{{j_{e - a} j_{d - c} - j_{d - a} j_{e - c}}}} \right)^k j_m\,,
\end{split}
\end{equation}
\begin{equation}
\begin{split}
&\sum\limits_{r = 0}^k {\binom kr\left( {\frac{{j_{e - a} j_{d - b} - j_{d - a} j_{e - b}}}{{j_{d - c} j_{e - b}  - j_{e - c} j_{d - b}}}} \right)^r j_{m + (b-a)k + (a-c) r} }\\
&\qquad = \left( {\frac{{j_{d - a} j_{e - c}  - j_{e - a} j_{d - c}}}{{j_{e - c} j_{d - b } - j_{d - c} j_{e - b}}}} \right)^k j_m\,,
\end{split}
\end{equation}
\begin{equation}
\begin{split}
&\sum\limits_{r = 0}^k {\binom kr\left( {\frac{{P_{d - c} P_{e - b}  - P_{e - c} P_{d - b}}}{{P_{d - a} P_{e - c}  - P_{e - a} P_{d - c}}}} \right)^r P_{m - (b-c) k + (b-a)r} }\\
&\qquad = \left( {\frac{{P_{d - a} P_{e - b}  - P_{e - a} P_{d - b}}}{{P_{d - a} P_{e - c}  - P_{e - a} P_{d - c}}}} \right)^k P_m\,,
\end{split}
\end{equation}
\begin{equation}
\begin{split}
&\sum\limits_{r = 0}^k {\binom kr\left( {\frac{{P_{e - a} P_{d - b} - P_{d - a} P_{e - b}}}{{P_{d - a} P_{e - c}  - P_{e - a} P_{d - c}}}} \right)^r P_{m + (a-b)k + (b-c) r} }\\
&\qquad = \left( {\frac{{P_{d - c} P_{e - b}  - P_{e - c} P_{d - b}}}{{P_{e - a} P_{d - c} - P_{d - a} P_{e - c}}}} \right)^k P_m\,,
\end{split}
\end{equation}
\begin{equation}
\begin{split}
&\sum\limits_{r = 0}^k {\binom kr\left( {\frac{{P_{e - a} P_{d - b} - P_{d - a} P_{e - b}}}{{P_{d - c} P_{e - b}  - P_{e - c} P_{d - b}}}} \right)^r P_{m + (b-a)k + (a-c) r} }\\
&\qquad = \left( {\frac{{P_{d - a} P_{e - c}  - P_{e - a} P_{d - c}}}{{P_{e - c} P_{d - b } - P_{d - c} P_{e - b}}}} \right)^k P_m\,,
\end{split}
\end{equation}
\begin{equation}
\begin{split}
&\sum\limits_{r = 0}^k {\binom kr\left( {\frac{{Q_{d - c} Q_{e - b}  - Q_{e - c} Q_{d - b}}}{{Q_{d - a} Q_{e - c}  - Q_{e - a} Q_{d - c}}}} \right)^r Q_{m - (b-c) k + (b-a)r} }\\
&\qquad = \left( {\frac{{Q_{d - a} Q_{e - b}  - Q_{e - a} Q_{d - b}}}{{Q_{d - a} Q_{e - c}  - Q_{e - a} Q_{d - c}}}} \right)^k Q_m\,,
\end{split}
\end{equation}
\begin{equation}
\begin{split}
&\sum\limits_{r = 0}^k {\binom kr\left( {\frac{{Q_{e - a} Q_{d - b} - Q_{d - a} Q_{e - b}}}{{Q_{d - a} Q_{e - c}  - Q_{e - a} Q_{d - c}}}} \right)^r Q_{m + (a-b)k + (b-c) r} }\\
&\qquad = \left( {\frac{{Q_{d - c} Q_{e - b}  - Q_{e - c} Q_{d - b}}}{{Q_{e - a} Q_{d - c} - Q_{d - a} Q_{e - c}}}} \right)^k Q_m
\end{split}
\end{equation}
and
\begin{equation}
\begin{split}
&\sum\limits_{r = 0}^k {\binom kr\left( {\frac{{Q_{e - a} Q_{d - b} - Q_{d - a} Q_{e - b}}}{{Q_{d - c} Q_{e - b}  - Q_{e - c} Q_{d - b}}}} \right)^r Q_{m + (b-a)k + (a-c) r} }\\
&\qquad = \left( {\frac{{Q_{d - a} Q_{e - c}  - Q_{e - a} Q_{d - c}}}{{Q_{e - c} Q_{d - b } - Q_{d - c} Q_{e - b}}}} \right)^k Q_m\,.
\end{split}
\end{equation}

\end{thm}

\end{document}